\documentclass{amsart}
\usepackage[utf8]{inputenc}

\numberwithin{equation}{section}

\usepackage[a4paper]{geometry}
\usepackage{pinlabel}
\usepackage{amssymb,amsfonts,amsmath}
\usepackage{comment} 
\usepackage[all,arc]{xy}
\usepackage{enumerate}
\usepackage{float}
\usepackage[caption = false]{subfig}
\usepackage{mathrsfs,mathtools}
\usepackage{todonotes,booktabs}
\usepackage{stmaryrd}
\usepackage{marvosym}
\usepackage{graphicx}
\usepackage{hyperref}
\usepackage{url}
\usepackage{pdflscape}
\usepackage{microtype}
\usepackage[all]{xy}
\usepackage{tikz-cd}
\usepackage[noabbrev, capitalise, nameinlink]{cleveref}
\usepackage{marginnote}
\usepackage{todonotes}

\newcommand*\circled[1]{\tikz[baseline=(char.base)]{
            \node[shape=circle,draw,inner sep=0.5pt] (char) {#1};}}

\definecolor{crimson}{rgb}{0.86, 0.08, 0.24}
\definecolor{darkcyan}{rgb}{0.0, 0.55, 0.55}

\hypersetup{
colorlinks={true},
linkcolor= {darkcyan},
citecolor= darkcyan}

\newtheorem{thm}{Theorem}[section]

\newtheorem*{thm*}{Theorem}
\newtheorem*{cor*}{Corollary}
\newtheorem*{prop*}{Proposition}
\newtheorem{cor}[thm]{Corollary}
\newtheorem*{notation*}{Notation}
\newtheorem{example}[thm]{Example}
\newtheorem{defn}[thm]{Definition}
\newtheorem*{defn*}{Definition}
\newtheorem{prop}[thm]{Proposition}
\newtheorem{lem}[thm]{Lemma}
\newtheorem{rem}[thm]{Remark}
\newtheorem{conj}[thm]{Conjecture}
\newtheorem*{conj*}{Conjecture}
\newtheorem*{quest*}{Question}
\newtheorem{quest}[thm]{Question}
\newtheorem{thmx}{Theorem}

\newcommand{\BPG}{BP^{(\!(G)\!)}}

\newcommand{\BPH}{BP^{(\!(H)\!)}}

\newcommand{\BPR}{BP_\mathbb{R}}
\newcommand{\BPCn}{BP^{(\!(C_{2^n})\!)}}
\newcommand{\BPQ}{BP^{(\!(Q_8)\!)}}
\newcommand{\MUR}{MU_\mathbb{R}}
\newcommand{\MUG}{MU^{(\!(G)\!)}}
\newcommand{\MUCn}{MU^{(\!(C_{2^n})\!)}}

\newcommand{\Z}{\underline{\mathbb{Z}}}
\newcommand{\BPCfour}{BP^{(\!(C_{4})\!)}}
\newcommand{\id}{\text{id}}
\newcommand{\E}{\mathcal{E}}
\def\WW{{{\mathbb{W}}}}

\DeclareMathOperator{\Ind}{Ind}
\DeclareMathOperator{\Gal}{Gal}
\DeclareMathOperator{\Aut}{Aut}

\DeclareMathOperator{\SliceSS}{\text{SliceSS}}
\DeclareMathOperator{\LSliceSS}{\text{LSliceSS}}

\DeclareMathOperator{\TateSS}{\text{TateSS}}
\DeclareMathOperator{\HFPSS}{\text{HFPSS}}

\author{Zhipeng Duan}\address{School of mathematical sciences, Nanjing Normal University}\email{zhipeng@njnu.edu.cn}

\author{Guchuan Li}\address{School of Mathematical Sciences, Peking University}\email{liguchuan@math.pku.edu.cn}

\author{XiaoLin Danny Shi}
\address{Department of Mathematics, University of Washington}
\email{dannyshixl@gmail.com}

\title{Vanishing lines in chromatic homotopy theory}

\begin{document}

\begin{abstract}
We show that at the prime 2, for any height $h$ and any finite subgroup $G \subset \mathbb{G}_h$ of the Morava stabilizer group, the $RO(G)$-graded homotopy fixed point spectral sequence for the Lubin--Tate spectrum $E_h$ has a strong horizontal vanishing line of filtration $N_{h, G}$, a specific number depending on $h$ and $G$.  It is a consequence of the nilpotence theorem that such homotopy fixed point spectral sequences all admit strong horizontal vanishing lines at some finite filtration.  Here, we establish specific bounds for them.  Our bounds are sharp for all the known computations of $E_h^{hG}$.  

Our approach involves investigating the effect of the Hill--Hopkins--Ravenel norm functor on the slice differentials.  As a result, we also show that the $RO(G)$-graded slice spectral sequence for $(N_{C_2}^{G}\bar{v}_h)^{-1}BP^{(\!(G)\!)}$ shares the same horizontal vanishing line at filtration $N_{h, G}$.  As an application, we utilize this vanishing line to establish a bound on the orientation order $\Theta(h, G)$, the smallest number such that the $\Theta(h, G)$-fold direct sum of any real vector bundle is $E_h^{hG}$-orientable.  
\end{abstract}

\maketitle

{\hypersetup{linkcolor=black}
\tableofcontents}

\section{Introduction}

\subsection{Motivation and main theorem}\label{subsec:1.1}

Chromatic homotopy theory originated with Quillen's groundbreaking observation of the relationship between the homotopy groups of the complex cobordism spectrum and the Lazard ring \cite{Quillencobordism}.  Subsequently, the work of Miller, Ravenel, and Wilson on periodic phenomena in the stable homotopy groups of spheres \cite{MillerRavenelWilson} and Ravenel's conjectures gave rise to what is now called the chromatic point of view.  This approach is a powerful tool for studying periodic phenomena in the stable homotopy category by analyzing the algebraic geometry of smooth one-parameter formal groups.  The moduli stack of formal groups has a stratification by height, and this stratification serves as an organizing framework for exploring large-scale phenomena in stable homotopy theory. 

Consider the Lubin--Tate spectrum $E(k,\Gamma_h)$ associated with a formal group law $\Gamma_h$ of height $h \geq 1$ over a finite field $k$ of characteristic $p$.  Up to an \'etale extension, these theories depend only on the height.  For the sake of clarity, we will implicitly choose a formal group law $\Gamma_h$ defined over $\mathbb{F}_p$ (i.e. the height-$h$ Honda formal group law) and a field $k$, and write $E_h = E(k, \Gamma_h)$.

The Chromatic Convergence Theorem of Hopkins and Ravenel \cite{ravenelorangebook} shows that the $p$-local sphere spectrum $S^0_{(p)}$ is the homotopy inverse limit of the chromatic tower
\[\cdots \longrightarrow L_{E_h}S^0 \longrightarrow \cdots \longrightarrow L_{E_1}S^0 \longrightarrow L_{E_0}S^0.\]
At each stage of this tower, $L_{E_h}S^0$ is the Bousfield localization of the sphere spectrum with respect to $E_h$.  These localizations can be inductively computed via the chromatic fracture square, which is the homotopy pullback square
\[\begin{tikzcd}
L_{E_h}S^0 \ar[r] \ar[d] & L_{K(h)}S^0 \ar[d] \\
L_{E_{h-1}}S^0 \ar[r] & L_{E_{h-1}}L_{K(h)}S^0.
\end{tikzcd}\]
Here, $K(h)$ is the height-$h$ Morava $K$-theory and $L_{K(h)}S^0$ is the $K(h)$-local sphere. 

Let $\mathbb{S}_h = \Aut_k(\Gamma_h)$, and define $\mathbb{G}_h = \mathbb{S}_h \rtimes \Gal(k/\mathbb{F}_p)$ to be the (big) Morava stabilizer group.  The continuous action of $\mathbb{G}_h$ on $\pi_* E_h$ can be refined to a unique $\mathbb{E}_\infty$-action of $\mathbb{G}_h$ on $E_h$ \cite{HopkinsMiller, goersshopkinsaction, LurieElliptic2}.  Devinatz and Hopkins \cite{devinatzHopkins} showed that $L_{K(h)}S^0 \simeq E_h^{h\mathbb{G}_h}$.  Furthermore, the $K(h)$-local $E_h$-based Adams spectral sequence for $L_{K(h)}S^0$ can be identified with the $\mathbb{G}_h$-homotopy fixed point spectral sequence for $E_h$: 
\[\E_2^{s, t} = H_c^s(\mathbb{G}_h, \pi_t E_h) \Longrightarrow \pi_{t-s} L_{K(h)}S^0.\]
Henn \cite{hennfiniteresolution} proposed that the $K(h)$-local sphere $L_{K(h)}S^0$ can be built up from spectra of the form $E_h^{hG}$, where $G$ is a finite subgroup of $\mathbb{G}_h$.  This construction has been explicitly realzied at heights 1 and 2 \cite{ GoerssK2-local, hennfiniteresolution,Agnesalgebraicduality, IrinaTopologicalresolution, henncentralrsolution}. 

From this point of view, the spectra $E_h^{hG}$ serve as the fundamental building blocks of the $p$-local stable homotopy category.  The homotopy groups $\pi_* E_h^{hG}$ also play a crucial role in detecting important families of elements in the stable homotopy groups of spheres \cite{Ravenel_arf, HHR, LiShiWaXu, BehrensMahowaldQuigley}.  Computation of these homotopy groups and understanding their Hurewicz images are central topics in chromatic homotopy theory. 

In this paper, we focus our attention at the prime $p=2$. Historically, describing the explicit action of $\mathbb{G}_h$ on $E_h$ has been challenging.  This limited our computations to heights 1 and 2 until the recent equivariant computational techniques introduced by Hill, Hopkins, and Ravenel \cite{HHR} (norms of Real bordism and the equivariant slice spectral sequence) and by Hahn and Shi \cite{realorientationDanny} (Real orientation).  These new techniques allowed us to compute $E_h^{hC_2}$ for all heights $h \geq 1$ \cite{realorientationDanny} and $E_4^{hC_4}$ \cite{HSWXC_4} at height 4.

The finite subgroups of $\mathbb{S}_h$ and $\mathbb{G}_h$ have been classified in \cite{HewettSubgroupMorava, Hewett2, bujard2012finite}.  To summarize this classification at the prime 2, let $h = 2^{n-1}m$, where $m$ is an odd number.  If $n\neq 2$, the maximal finite $2$-subgroups of $\mathbb{S}_h$ are isomorphic to $C_{2^n}$, the cyclic group of order $2^n$.  When $n=2$, the maximal finite $2$-subgroups of $\mathbb{S}_h$ are isomorphic to the quaternion group $Q_8$. Furthermore, the group $\mathbb{G}_h$ contains a subgroup of order two, corresponding to the automorphism $[-1]_{\Gamma_h}(x)$ of $\Gamma_h$.  This $C_2$-subgroup is central in $\mathbb{G}_h$.  All the finite subgroups $G \subset \mathbb{G}_h$ we consider in this paper will contain this central $C_2$-subgroup.  

To state our main result, note that based on the classification provided above, for any $G \subset \mathbb{G}_h$ a finite subgroup, a 2-Sylow subgroup $H$ of $G \cap \mathbb{S}_h$ is isomorphic to either $C_{2^n}$ or $Q_8$.

\begin{defn} \label{def:NhG}\rm
For $h > 0$ and $G \subset \mathbb{G}_h$ a finite subgroup, let $H$ be a 2-Sylow subgroup of $K = G \cap \mathbb{S}_h$.  Define $N_{h,G}$ to be the positive integer $N_{h,H}$, where 
\begin{align*}
N_{h, C_{2^n}} &:= 2^{h+n} - 2^n + 1, \\
N_{h, Q_8} &:= 2^{h+3}-7.
\end{align*}
\end{defn}

The main result of this paper is the following: 

\begin{thmx}[Horizontal Vanishing Line]\label{thm:introThm1}
For any height $h$ and any finite subgroup $G \subset \mathbb{G}_h$, there is a strong horizontal vanishing line of filtration $N_{h,G}$ in the $RO(G)$-graded homotopy fixed point spectral sequence for $E_h$. 
\end{thmx}

Recall that having a strong horizontal vanishing line of filtration $N_{h, G}$ means that the spectral sequence collapses after the $\E_{N_{h, G}}$-page, with no surviving elements of filtration greater than or equal to $N_{h, G}$ at the $\E_\infty$-page.

The motivation behind \cref{thm:introThm1} is as follows: classically, the Nilpotence Theorem of Devinatz, Hopkins, and Smith \cite{Nilpotence1, Nilpotence2} ensures that the homotopy fixed point spectral sequences of the Lubin--Tate theories $E_h$ all have strong horizontal vanishing lines at \textit{some} finite filtration (see \cite[Section~5]{devinatzHopkins} and \cite[Section~2.3]{BeaudryGoerssHenn}).  While theoretically valuable, this existence result alone cannot be used for computations.  Without knowledge of the specific location of the vanishing line, it cannot aid in proving specific differentials.

The recent computations by Hill, Shi, Wang, and Xu have demonstrated the utility of having a bound for the strong horizontal vanishing line in equivariant computations of Lubin--Tate theories.  In their work \cite{HSWXC_4}, they first re-analyzed the slice spectral sequence for $\BPCfour \langle 1 \rangle$ (a connective model of $E_2$ with a $C_4$-action), and established a horizontal vanishing line of filtration 16.  They also proved that every class on or above this line must vanish on or before the $\E_{13}$-page \cite[Theorem~3.17]{HSWXC_4}.  This result allowed them to provide a more concise proof of all the Hill--Hopkins--Ravenel slice differentials presented in \cite{HHRC_4}. 

In the subsequent case, when studying the slice spectral sequence for $\BPCfour \langle 2 \rangle$ (a connective model of $E_4$ with a $C_4$-action), a similar phenomenon was observed. There exists a horizontal vanishing line at filtration 96, and every class situated on or above this line must vanish on or before the $\E_{61}$-page.  This theorem is referred to as the Vanishing Theorem \cite[Theorem~9.2]{HSWXC_4}, and it serves as a crucial tool in establishing many of the higher slice differentials.  

The strong vanishing lines established in \cref{thm:introThm1} will significantly facilitate future computations involving Lubin--Tate theories and norms of Real bordism theories.

\subsection{Main results and outline of the paper}
We will now give a more detailed summary of our results and describe the contents of this paper. 

In \cref{section:preliminaries}, we recall some basic facts of our spectral sequences of interest. The classical Tate diagram induces a Tate diagram of spectral sequences 
\[\xymatrix{
    \text{HOSS(X)}\ar[r]\ar[d]^{=}&\text{SliceSS}(X)\ar[d] \ar[r]& \text{LSliceSS}(X)\ar[d]\\
    \text{HOSS}(X)\ar[r]& \text{HFPSS}(X)\ar[r]&\text{TateSS}(X).
    }
  \]
  The interactions between these spectral sequences will be crucial for proving our main theorem. 

We will also recall the spectrum $\BPG$, its slice filtration, and some special classes on the $\E_2$-page of its slice spectral sequence.  We prove all the differentials in the $C_2$-slice spectral sequence for $i_{C_2}^*\BPG$ when $G = C_{2^n}$ and $Q_8$ (\cref{thm:slicediffthm}).  While not stated elsewhere, this is a straightforward consequence of \cite[Theorem~9.9]{HHR}. 

In \cref{sec:Comparison}, we prove comparison theorems between the slice spectral sequence, the homotopy fixed point spectral sequence, and the Tate spectral sequence.  These comparisons are based on the maps
\[\SliceSS(X) \longrightarrow \HFPSS(X) \longrightarrow \TateSS(X)\]
extracted from the Tate diagram of spectral sequences above.  It is worth noting that prior works by Ullman \cite{Ullman} and B\"{o}ckstedt--Madsen \cite{IbMadsenTate} have shown that both maps induce isomorphisms within specific ranges in the integer-graded spectral sequence.  For our purposes, we extend these isomorphism regions to the $RO(G)$-graded pages.  
\begin{thmx}[\cref{defn:tauV} and \cref{prop:isoslicehfp}] \label{thm:introThm2}
For $V \in RO(G)$, let 
\[\tau(V): = \min_{\{e\} \subsetneq H \subset G} |H| \cdot \dim V^H.\]  
The map from the $RO(G)$-graded slice spectral sequence to the $RO(G)$-graded homotopy fixed point spectral sequence induces an isomorphism on the $\E_2$-page for pairs $(V, s)$ that satisfy the inequality 
\[\tau(V-s-1) > |V|.\]
Furthermore, this map induces a one-to-one correspondence between the differentials within this isomorphism region.  
\end{thmx}

The proof of \cref{thm:introThm2} relies on the main result in Hill--Yarnall \cite[Theorem~A]{HillYarnall}, which establishes a relationship between the slice connectivity of an equivariant spectrum and the connectivity of its geometric fixed points.  

As for the map from the homotopy fixed point spectral sequence to the Tate spectral sequence, the classical analysis almost generalizes immediately to give an $RO(G)$-graded isomorphism region.  

\begin{thmx}[\cref{prop:isohfptate}] \label{thm:introThm3}
The map from the $RO(G)$-graded homotopy fixed point spectral sequence to the $RO(G)$-graded Tate spectral sequence induces an isomorphism on the $\E_2$-page for classes in filtrations $s >0$, and a surjection for classes in filtration $s = 0$.  Furthermore, there is a one-to-one correspondence between differentials whose sources are of nonnegative filtration. 
\end{thmx}

In \cref{sec:NormStructure}, we give a brief summary of the norm structure in equivariant spectral sequences.  This structure plays a pivotal role in deducing the fate of specific classes in the $G$-equivariant spectral sequence based on information from the $H$-equivariant spectral sequence, where $H \subset G$ is a subgroup (\cref{thm:normdiff}). 

In \cref{sec:TateVanishing}, we analyze the Tate spectral sequence for $E_h$ and prove the following theorem.
 
\begin{thmx}[Tate Vanishing, \cref{theorem:TateVanishing}] \label{thm:introThm4}
For any height $h$ and any finite subgroup $G \subset \mathbb{G}_h$, all the classes in the $RO(G)$-graded Tate spectral sequence for $E_h$ vanish after the $\E_{N_{h, G}}$-page.  Here, $N_{h, G}$ is defined as in \cref{def:NhG}. 
\end{thmx}

Note that at any prime $p$, Mathew and Meier have shown that the map $E_h^{hG} \to E_h$ is a faithful $G$-Galois extension whenever $G \subset \mathbb{G}_h$ is a finite subgroup \cite[Example~6.2]{MathewMeier}.  This implies that the Tate spectrum $E_h^{tG}$ is contractible \cite[Proposition~6.3.3]{RognesGalois}.  Consequently, all the classes in the Tate spectral sequence for $E_h$ must eventually vanish.  \cref{thm:introThm4} provides a concrete bound for the page number at which this vanishing occurs when $p = 2$.

To prove \cref{thm:introThm4}, we use the $G$-equivariant orientation from $\BPG$ to $E_h$, as given by \cite{realorientationDanny}.  This orientation map factors through $(N_{C_2}^G \bar{v}_h)^{-1} \BPG$: 
\[\begin{tikzcd}
\BPG \ar[r] \ar[d] & E_h \\ 
(N_{C_2}^G \bar{v}_h)^{-1} \BPG \ar[ru]
\end{tikzcd}\]
This induces a map of the corresponding Tate spectral sequences: 
\[G\text{-}\TateSS((N_{C_2}^G \bar{v}_h)^{-1} \BPG) \longrightarrow G\text{-}\TateSS(E_h). \]
Equipped with the results discussed in the previous sections, we first transport the differentials from the $C_2$-slice spectral sequence for $i_{C_2}^*(N_{C_2}^G \bar{v}_h)^{-1}\BPG$ to the $C_2$-Tate spectral sequence for $i_{C_2}^*(N_{C_2}^G \bar{v}_h)^{-1}\BPG$ using the one-to-one correspondences established in \cref{sec:Comparison}.  We then use the norm structure to deduce that the unit class in the $G$-Tate spectral sequence for $(N_{C_2}^G \bar{v}_h)^{-1}\BPG$ must be killed on or before the $\E_{N_{h, G}}$-page.  By naturality, the unit class in the $G$-Tate spectral sequence for $E_h$ must also be killed on or before the $\E_{N_{h, G}}$-page.  This leads to the vanishing of all other classes beyond this point by the multiplicative structure. 

Our proof of \cref{thm:introThm4} applies in general to give a similar vanishing theorem for any $(N_{C_2}^G \bar{v}_h)^{-1}\BPG$-module. 
\begin{cor} [\cref{rem:TateVanishingModule}] \label{thm:introThm4.5}
Let $M$ be a $(N_{C_2}^G \bar{v}_h)^{-1}\BPG$-module.  All the classes in the $RO(G)$-graded Tate spectral sequence for $M$ vanish after the $\E_{N_{h, G}}$-page. 
\end{cor}

In \cref{sec:HFPSSVanishing}, we analyze the homotopy fixed point spectral sequence for $E_h$ and prove \cref{thm:introThm1} (\cref{thm:vanishingallfinite}).  The proof of \cref{thm:introThm1} is by using the comparison theorem (\cref{thm:introThm3}) between the homotopy fixed point spectral sequence and the Tate spectral sequence, combined with the Tate vanishing theorem (\cref{thm:introThm4}) in the Tate spectral sequence.  Our proof also applies to show that the same strong horizontal vanishing line exists in the homotopy fixed point spectral sequence for any $(N_{C_2}^G \bar{v}_h)^{-1} \BPG$-module. 

\begin{cor}[\cref{rem:HFPSSVanishingModule}]
For any $(N_{C_2}^G \bar{v}_h)^{-1} \BPG$-module $M$, there is a strong horizontal vanishing line of filtration $N_{h, G}$ in the $RO(G)$-graded homotopy fixed point spectral sequence for $M$.
\end{cor}

\begin{cor}[Uniform Vanishing, \cref{cor:uniformvanishing}]
For any $K(h)$-local finite spectrum $Z$, the homotopy fixed point spectral sequence 
\[H^s(G, E_tZ) \Longrightarrow \pi_{t-s}(E^{hG} \wedge Z)\]
has a strong horizontal vanishing line of filtration $N_{h, G}$. 
\end{cor}

In \cref{sec:VanishingSlice}, we prove the existence of horizontal vanishing lines in the slice spectral sequence.  

\begin{thmx}[\cref{thm:sliceVanishing}] \label{thm:introThm6}
When $G = C_{2^n}$ or $Q_8$, the $RO(G)$-graded slice spectral sequence for any $(N_{C_2}^G \bar{v}_h)^{-1}\BPG$-module $M$ admits a horizontal vanishing line of filtration $N_{h, G}$.  
\end{thmx}

In particular, \cref{thm:introThm6} implies that there will be a horizontal vanishing line of filtration 121 in the $C_8$-slice spectral sequence for $\Omega_\mathbb{O}$, the detection spectrum of Hill--Hopkins--Ravenel that detects all the Kervaire invariant one elements \cite{HHR}.

It is interesting to note that when $G = Q_8$, even though there is no knowledge of the slice filtration of $\BPQ$ yet, \cref{thm:introThm6} still applies to show that the slice spectral sequences of $(N_{C_2}^{Q_8}\bar{v}_h)^{-1} \BPQ$-modules all have horizontal vanishing lines of filtration $N_{h, Q_8}$.

Finally, in \cref{section:application}, we present an application of \cref{thm:introThm1} in the study of $E_h^{hG}$-orientations of real vector bundles.  For $h \geq 1$ and $G \subseteq \mathbb{G}_h$ a finite subgroup, let $\Theta(h, G)$ be the smallest number $d$ such that the $d$-fold direct sum of any real vector bundle is $E_h^{hG}$-orientable.  At the prime $p = 2$ and $G = C_2$, Kitchloo and Wilson \cite{KitchlooWilson} have studied $E_h^{hC_2}$-orientations.  When $G = C_p$, Bhattacharya and Chatham \cite{Hoodoriented} have studied $E_{k(p-1)}^{hC_p}$-orientations at all primes.   

\begin{thmx}[\cref{theorem:orientationC2m}]
For any height $h$ and any finite subgroup $G\subset \mathbb{G}_h$, let $K=G\cap \mathbb{S}_h$, $H$ be a $2$-Sylow subgroup of $K$, and define $d= 2\cdot|K|\cdot |H|^{\frac{N_{h,H}-1}{2}}$. Then the $d$-fold direct sum of any real vector bundle is $E_h^{hG}$-orientable.
\end{thmx}

\subsection{Open questions and further directions}
\subsubsection*{Sharpness of the strong horizontal vanishing lines}

For all known computations, the bounds established in \cref{thm:introThm1} for the strong horizontal vanishing lines are sharp when the 2-Sylow subgroup of $K = G \cap \mathbb{S}_h$ is cyclic.  More specifically, when $G = C_2$, the strong horizontal vanishing line in the homotopy fixed point spectral sequence for $E_h^{hC_2}$ is at filtration exactly $2^{h+1}-1$.  When $G = C_4$, the strong horizontal vanishing lines in the homotopy fixed point spectral sequences for $E_2^{hC_4}$ and $E_4^{hC_4}$ are at filtrations exactly $13$ and $61$. 

When the 2-Sylow subgroup of $K$ is isomorphic to $Q_8$, Bauer's computation of $\texttt{tmf}$ \cite{bauertmf} implies that the strong horizontal vanishing line in the $Q_8$-homotopy fixed point spectral sequence for $E_2$ is at filtration 23.  This value is lower than the bound provided by our theorem, which is 25.  In \cite{DuanKongLiLuWangQ8}, the value $N_{h, Q_8}$ has been further reduced from $2^{h+3}-7$ to $2^{h+3}-9$.  \cref{thm:introThm1}, combined with this new improvement, yields the sharpest bounds for the strong horizontal vanishing lines across all known computations.

\begin{conj} \label{conj:introConjBoundSharp}
The bounds established in \cref{thm:introThm1} for the strong horizontal vanishing lines are sharp.
\end{conj}

Devinatz and Hopkins \cite{devinatzHopkins} have also proved that for the big Morava stabilizer group $\mathbb{G}_h$, the homotopy fixed point spectral sequence for $E_h^{h\mathbb{G}_h}$ admits a strong vanishing line at some finite filtration.

\begin{conj} \label{conj:BigGhVanishingLine}
The homotopy fixed point spectral sequence for $E_h^{h\mathbb{G}_h}$ admits a strong vanishing line at filtration $(h^2+N)$, where $N:=\max\{N_{h,G}\mid G \subset \mathbb{G}_h \text{ finite}\}.$
\end{conj}

An intuitive reason for the bound $(h^2 + N)$ in \cref{conj:BigGhVanishingLine} is as follows: by the philosophy of finite resolutions, there should be a resolution of $E_h^{h\mathbb{G}_h}$ built from the finite fixed points $\{E_h^{hG} \mid G \subset \mathbb{G}_h \text{ finite}\}$, and this resolution should have length $h^2$ because this number is the virtual cohomological dimension of $\mathbb{S}_h$.  Analyzing the associated tower of spectral sequences produces the conjectural bound. 

\subsubsection*{Odd primes}

\begin{quest}
At odd primes, what is the filtration of the strong horizontal vanishing line in the homotopy fixed point spectral sequence for $E_h^{hG}$?
\end{quest}

Note that, as a consequence of the classification of finite subgroups of $\mathbb{S}_h$ at odd primes \cite{HewettSubgroupMorava}, when $h = p^{n-1}(p-1)m$, there is a cyclic subgroup of order $p^n$ in $\mathbb{S}_h$.  The authors believe that once a comprehensive understanding of the $C_p$-homotopy fixed point spectral sequence for $E_h$ is achieved, the arguments presented in this paper can be employed analogously to establish a bound for the strong horizontal vanishing line in $\HFPSS(E_h^{hG})$ that is applicable to any height $h$ and any finite subgroup $G \subset \mathbb{G}_h$ containing $C_p$. 

\subsubsection*{Horizontal vanishing lines for connective theories}

When $G = C_{2^n}$, the Hill--Hopkins--Ravenel quotient $(N_{C_2}^G\bar{v}_h)^{-1}\BPCn \langle m \rangle$ is a $(N_{C_2}^G\bar{v}_h)^{-1}\BPCn$-module, and there is a horizontal vanishing line in its $RO(C_{2^n})$-graded slice spectral sequence at filtration $N_{h, C_{2^n}}$ by \cref{thm:introThm6}.

However, without inverting the class $(N_{C_2}^G \bar{v}_h)$, there is no horizontal vanishing line in the $RO(C_{2^n})$-graded slice spectral sequence for the connective theory $\BPCn \langle m \rangle$.  This is because we have elements of arbitrarily high filtrations on the $\E_\infty$-page.  For example, the tower $\{a_\sigma^k \mid k \geq 1\}$ contains classes of arbitrarily high filtrations that survive to the $\E_\infty$-page.

Interestingly, computations of $\texttt{tmf}$, $\BPR\langle n \rangle$, $\BPCfour \langle 1 \rangle$, and $\BPCfour \langle 2 \rangle$ suggest the presence of horizontal vanishing lines in the integer-graded slice spectral sequence for the connective theories \cite{bauertmf, HuKriz, HHRC_4, HSWXC_4}, with filtrations matching the filtrations for the vanishing lines of the periodic theories.   

\begin{conj}
There is a horizontal vanishing line of filtration $N_{h, C_{2^n}}$ in the integer-graded slice spectral sequence for $\BPCn \langle m \rangle$.
\end{conj}

\subsection{Acknowledgements}
The authors would like to thank Agn\`{e}s Beaudry, Prasit Bhattacharya, Hood Chatham, Paul Goerss, Mike Hill, Tyler Lawson, Yunze Lu, Peter May, Zhouli Xu, Mingcong Zeng, and Foling Zou for helpful conversations.  We would like to thank Guozhen Wang for comments on an earlier draft of our paper and answering our numerous questions.  We would also like to thank the anonymous referee for the many helpful comments and suggestions.  The third author is supported in part by NSF
Grant DMS-2313842.

\section{Preliminaries}\label{section:preliminaries}

In this section, we will discuss the spectral sequences that are of interest to us.  We will also collect certain facts about these spectral sequences that we will need in the later sections.  

Let $X$ be a $G$-spectrum, and let $P^{\bullet}X$ be the slice tower of $X$.  The Tate diagram 
\begin{equation*}
    \begin{gathered}\xymatrix{EG_+\wedge X \ar[r]\ar[d]^{\simeq}& X\ar[r]\ar[d] & \widetilde{E}G\wedge X \ar[d]\\
    EG_+\wedge F(EG_+,X)\ar[r]& F(EG_+,X)\ar[r]  & \widetilde{E}G\wedge F(EG_+,X)
    }
    \end{gathered}
\end{equation*}
induces a diagram of towers: 
\begin{equation*}
    \begin{gathered}\xymatrix{EG_+\wedge P^{\bullet}X \ar[r]\ar[d]^{\simeq}& P^{\bullet}X\ar[r]\ar[d] & \widetilde{E}G\wedge P^{\bullet}X \ar[d]\\
    EG_+\wedge F(EG_+,P^{\bullet}X)\ar[r]& F(EG_+,P^{\bullet}X)\ar[r]  & \widetilde{E}G\wedge F(EG_+,P^{\bullet}X).
    }
    \end{gathered}
\end{equation*}

This diagram of towers further induces a Tate diagram of spectral sequences 
\begin{equation} \label{diag:TateDiagramSS}
    \begin{gathered}\xymatrix{
    \text{HOSS(X)}\ar[r]\ar[d]^{=}&\text{SliceSS}(X)\ar[d]^{\circled{1}}\ar[r]& \text{LSliceSS}(X)\ar[d]\\
    \text{HOSS}(X)\ar[r]& \text{HFPSS}(X)\ar[r]^{\circled{2}} &\text{TateSS}(X).
    }
     \end{gathered}
\end{equation}

All the spectral sequences in (\ref{diag:TateDiagramSS}) are $RO(G)$-graded spectral sequences.  We pause to briefly discuss notations:
\begin{enumerate}
\item The spectral sequence associated with the tower $\{EG_+ \wedge P^\bullet X\}$ is the \textit{homotopy orbit spectral sequence} (HOSS) of $X$.  It is a third and fourth quadrant spectral sequence, and it converges to $\underline{\pi}_\star EG_+ \wedge X$.  In the integer-graded page at the $(G/G)$-level, the spectral sequence converges to $\pi_*^G EG_+ \wedge X = \pi_*X_{hG}$.

\item The spectral sequence associated with the tower $\{P^\bullet X\}$ is the \textit{slice spectral sequence} (SliceSS) of $X$.  It is a first and third quadrant spectral sequence, and it converges to $\underline{\pi}_\star X$.  In the integer graded page at the $(G/G)$-level, the spectral sequence converges to $\pi_*^G X = \pi_* X^G$. 

\item Following the treatment of \cite{MeierShiZengSegal}, the spectral sequence associated with the tower $\{\widetilde{E}G \wedge P^{\bullet}X\}$ is called the \textit{localized slice spectral sequence for $X$} and is denoted by $\LSliceSS(X)$.  It converges to $\underline{\pi}_\star \widetilde{E}G \wedge X$. 

\item The spectral sequence associated with the tower $\{F(EG_+, P^\bullet X)\}$ is the \textit{homotopy fixed point spectral sequence} (HFPSS) of $X$.  It is a first and second quadrant spectral sequence, and it converges to $\underline{\pi}_\star F(EG_+, X)$.  In the integer-graded page at the $(G/G)$-level, the spectral sequence converges to $\pi_*^G F(EG_+, X) = \pi_* X^{hG}$. 

\item The spectral sequence associated with the tower $\{\widetilde{E}G \wedge F(EG_+, P^\bullet X)\}$ is the \textit{Tate spectral sequence} (TateSS) of $X$.  It has classes in all four quadrants, and it converges to $\underline{\pi}_\star \widetilde{E}G \wedge F(EG_+, X)$.  In the integer-graded page at the $(G/G)$-level, the spectral sequence converges to $\pi_*^G\widetilde{E}G \wedge F(EG_+, X) = \pi_* X^{tG}$.
\end{enumerate}

Let $\rho_2$ denote the regular $C_2$-representation.  In \cite{BHSZModel}, it is shown that there are generators 
\[\bar{t}_i \in \pi_{(2^i-1)\rho_{2}}^{C_2} \BPCn \]
such that 
\[\pi_{*\rho_2}^{C_2} \BPCn \cong \mathbb{Z}_{(2)} [C_{2^n} \cdot \bar{t}_1, C_{2^n} \cdot \bar{t}_2, \ldots, ].\]
For a precise definitions of these generators, see formula (1.3) in \cite{BHSZModel} (also see \cite[Section~5]{HHR} for analogous generators in $\pi_{*\rho_2}^{C_2}\MUCn$).  For $\BPR$, we will denote the $\bar{t}_i$-generators as $\bar{v}_i$, as their restrictions give a set of generators $v_i \in \pi_{2(2^i-1)}BP$ for $\pi_*BP$. 

Similar to the treatment of $\MUCn$ in \cite{HHR}, we can build an equivariant refinement 
\[S^0[C_{2^n} \cdot \bar{t}_1, C_{2^n} \cdot \bar{t}_2, \ldots] \longrightarrow \BPCn\]
from which we can apply the Slice Theorem \cite[Theorem~6.1]{HHR} to show that the slice associated graded of $\BPCn$ is the graded spectrum 
\[H\Z[C_{2^n} \cdot \bar{t}_1, C_{2^n} \cdot \bar{t}_2, \ldots].\]
Here, the degree of a summand corresponding to a monomial in the $\bar{t}_i$-generators and their conjugates is the underlying degree.  

As a consequence, the slice spectral sequence for the $RO(C_{2^n})$-graded homotopy groups of $\BPCn$ has $\E_2$-term the $RO(C_{2^n})$-graded homotopy of $H\Z[C_{2^n} \cdot \bar{t}_1, C_{2^n} \cdot \bar{t}_2, \ldots]$.  To compute this, note that $S^0[C_{2^n} \cdot \bar{t}_1, C_{2^n} \cdot \bar{t}_2, \ldots]$ can be decomposed into a wedge sum of slice cells of the form 
\[{C_{2^n}}_+ \wedge_{H_p} S^{\frac{|p|}{|H_p|}\rho_{H_p}},\]
where $p$ ranges over a set of representatives for the orbits of monomials in the $\gamma^j \bar{t}_i$-generators, and $H_p \subset C_{2^n}$ is the stabilizer of $p \pmod{2}$.  Therefore, it suffices to compute the equivariant homology groups of the representations spheres $S^{\frac{|p|}{|H_p|}\rho_{H_p}}$ with coefficients in the constant Mackey functor $\Z$. 

We recall some distinguished elements in the $RO(G)$-graded homotopy groups that we will need in order to name the relevant classes on the $\E_2$-page of the slice spectral sequence (see \cite[Section~3.4]{HHR} and \cite[Section~2.2]{HSWXC_4}).  

\begin{defn}\rm
Let $V$ be a $G$-representation.  We will use $a_{V}:S^0\rightarrow S^{V}$ to denote its \textit{Euler class}.  This is an element in $\pi_{-V}^GS^0$.  We will also denote its Hurewciz image in $\pi_{-V}^G H\underline{\mathbb{Z}}$ by $a_V$. 
\end{defn}
If the representation $V$ has nontrivial fixed points (i.e. $V^{G}\neq \{0\}$), then $a_V=0$.  Moreover, for any two $G$-representations $V$ and $W$, we have the relation $a_{V\oplus W}=a_{V}a_{W}$ in $\pi^{G}_{-V-W}(S^0)$.

\begin{defn}\rm
Let $V$ be an oriented $G$-representation.  Then the orientation for $V$ gives an isomorphism $H_{|V|}^{G} (S^V; \underline{\mathbb{Z}}) \cong \mathbb{Z}$.  In particular, the restriction map 
\[
H^{G}_{|V|}(S^V,\underline{\mathbb{Z}})\longrightarrow H_{|V|}(S^{|V|},\mathbb{Z})
\]
is an isomorphism.  Let $u_V\in H_{|V|}^{G}(S^V; \underline{\mathbb{Z}})$ be the generator that maps to $1$ under this restriction isomorphism.  The class $u_V$ is called the \textit{orientation class} of $V$.
\end{defn}

The orientation class $u_V$ is stable in $V$.  More precisely, if 1 is the trivial representation, then $u_{V\oplus 1}=u_{V}$.  Moreover, if $V$ and $W$ are two oriented $G$-representations, then $V \oplus W$ is also oriented, and $u_{V\oplus W}=u_{V}u_{W}$.  

The Euler class $a_V$ and the orientation class $u_V$ behave well with respect to the Hill--Hopkins--Ravenel norm functor.  More precisely, for $H \subset G$ a subgroup and $V$ a $H$-representation, we have the equalities
\begin{align}
N_H^G(a_V)&=a_{\Ind V} \label{eq:normavuv1}\\ 
u_{\Ind |V|}N_H^G (u_V)&=u_{\Ind V} \label{eq:normavuv2} 
\end{align}
where $\Ind V=\Ind_H^G V$ is the induced representation.

When $G = C_{2^n}$, let $\lambda_{i}$, $1 \leq i \leq n$ denote the 2-dimensional real $C_{2^n}$-representation corresponding to rotation by $\left(\frac{2\pi}{2^{i}}\right)$.  In particular, when $i = 1$, the representation $\lambda_1$ corresponds to rotation by $\pi$ and thus equals to $2\sigma$, where $\sigma$ is the real sign representation of $C_{2^n}$.  When localized at 2, the representations that will be relevant to us are $1$, $\sigma$, $\lambda_2$, $\lambda_3$, $\ldots$, $\lambda_n$.

When $G = Q_8$, $RO(Q_8) = \mathbb{Z}\{1, \sigma_i, \sigma_j, \sigma_k, \mathbb{H}\}$.  The representations $\sigma_i$, $\sigma_j$, and $\sigma_k$ are one-dimensional representations whose kernels are $\langle i \rangle$, $\langle j \rangle$, and $\langle k \rangle$, respectively.  The representation $\mathbb{H}$ is a four-dimensional irreducible representation, obtained by the action of $Q_8$ on the quaternion algebra $\mathbb{H} = \mathbb{R} \oplus \mathbb{R}i \oplus \mathbb{R}j \oplus \mathbb{R}k$ by left multiplication.

For $h \geq 1$, let $\bar{v}_h \in \pi_{(2^h-1)\rho_2}^{C_2} \BPG$ denote the images of $\bar{v}_h$-generators under the map 
\[\BPR \longrightarrow i_{C_2}^* \BPG,\]
which is inclusion into the first factor.  The following theorem describes all the differentials in the slice spectral sequence for $i_{C_2}^* \BPG$.  

\begin{thm}\label{thm:slicediffthm}
Let $G = C_{2^n}$ or $Q_8$.  In the $C_2$-slice spectral sequence for $i^{*}_{C_2}\BPG$, the differentials are generated under multiplicative structures by the differentials 
\[d_{2^{h+1}-1}(u_{2\sigma_2}^{2^{h-1}})=\bar{v}_h a_{\sigma_2}^{2^{h+1}-1}, \,\,\, h \geq 1.\]
\end{thm}

\begin{proof}
When $G = C_2$, the claim is immediate from the Slice Differential Theorem of Hill--Hopkins--Ravenel \cite[Theorem~9.9]{HHR}.  When $G$ is $C_{2^n}$ or $Q_8$ for $n \geq 2$, the $C_2$-restriction of $\BPG$ is a smash product of $(|G|/2)$-copies of $\BPR$.  In this case, we have a complete understanding of its $C_2$-slices and the $\E_2$-page of its $C_2$-slice spectral sequence.

The unit map $BP_\mathbb{R} \to i_{C_2}^*\BPG$ induces a map 
\begin{equation}\label{eq:BPRtoBPG}
\SliceSS(BP_\mathbb{R}) \longrightarrow \SliceSS(i_{C_2}^*\BPG)
\end{equation}
of $C_2$-slice spectral sequences.  We will proceed by using induction on $h$.  For the base case, when $h =1$, we have the $d_3$-differential 
\[d_3(u_{2\sigma_2}) = \bar{v}_1 a_{\sigma_2}^3\]
in $\SliceSS(\BPR)$.  Under the map (\ref{eq:BPRtoBPG}), the source is mapped to $u_{2\sigma_2}$ and the target is mapped to $\bar{v}_1 a_{\sigma_2}^3$.  By naturality, $\bar{v}_1 a_{\sigma_2}^3$ must be killed by a differential of length at most 3.  Since the lowest possible differential length is 3 by degree reasons, the $d_3$-differential 
\[d_3(u_{2\sigma_2}) = \bar{v}_1 a_{\sigma_2}^3\]
must occur in $\SliceSS(i_{C_2}^*\BPG)$.  Multiplying this differential by permanent cycles determines the rest of the $d_3$-differentials.  For degree reasons, these are all the $d_3$-differentials.

Suppose now that the induction hypothesis holds for all $1 \leq k \leq h-1$.  For degree reasons, after the $d_{2^h-1}$-differentials, the next possible differential is of length $d_{2^{h+1}-1}$.  In $\SliceSS(\BPR)$, consider the differential 
\[d_{2^{h+1}-1}(u_{2\sigma_2}^{2^{h-1}}) = \bar{v}_h a_{\sigma_2}^{2^{h+1}-1}.\]
The map (\ref{eq:BPRtoBPG}) sends both the source and the target of this differential to nonzero classes of the same name in $\SliceSS(i_{C_2}^*\BPG)$.  By naturality, the image of the target, $\bar{v}_ha_{\sigma_2}^{2^{h+1}-1}$, must be killed by a differential of length at most $2^{h+1}-1$.  For degree reasons, it is impossible for this class to be killed by a differential of length smaller than $2^{h+1}-1$.  It follows that the differential 
\[d_{2^{h+1}-1}(u_{2\sigma_2}^{2^{h-1}}) = \bar{v}_h a_{\sigma_2}^{2^{h+1}-1}\]
exists in $\SliceSS(i_{C_2}^*\BPG)$.  The rest of the $d_{2^{h+1}-1}$-differentials are determined by multiplying this differential with permanent cycles.  After these differentials, there is no room for other $d_{2^{h+1}-1}$-differentials by degree reasons.  This completes the induction step.  
\end{proof}

\begin{rem}\rm
We are grateful to Mike Hill for sharing the following argument, which directly shows that $\SliceSS(i_{C_2}^*\MUG)$ (and consequently $\SliceSS(i_{C_2}^*\BPG)$) is completely determined by $\SliceSS(\BPR)$.  This offers an alternative and shorter proof for \cref{thm:slicediffthm}.  The Thom isomorphism provides an equivalence 
\[i_{C_2}^* \MUG \simeq MU_\mathbb{R} \wedge {({BU_\mathbb{R}}_+)}^{\wedge(|G|/2-1)}, \]
where $BU_\mathbb{R}$ is the $C_2$-space $BU$ equipped with the complex conjugation action.  Since $\MUR$ is Real oriented, the right-hand side splits as $\MUR \wedge A$, where $A$ is a wedge of suspensions of regular representation spheres.  Therefore, $\SliceSS(i_{C_2}^*\MUG)$ also splits as a wedge of suspensions of $\SliceSS(\MUR)$ by regular representations.  When localized at 2, $\SliceSS(\MUR)$ further splits as a wedge of suspensions of $\SliceSS(\BPR)$.  It follows that the differentials in $\SliceSS(\BPR)$ completely determine the differentials in $\SliceSS(i_{C_2}^*\MUG)$, and consequently, the differentials in $\SliceSS(i_{C_2}^*\BPG)$.
\end{rem}

\section{Comparison of spectral sequences}\label{sec:Comparison}
In \cite{Ullman} and \cite{IbMadsenTate}, it is shown that the maps $\circled{1}$ and $\circled{2}$ in (\ref{diag:TateDiagramSS}) induce isomorphisms in a certain range in the integer-graded page.  For our purposes, we will extend their integral-graded isomorphism ranges to $RO(G)$-graded isomorphism ranges. 

\begin{defn} \label{defn:tauV}\rm
For $V \in RO(G)$, let 
\[\tau(V): = \min_{\{e\} \subsetneq H \subset G} |H| \cdot \dim V^H. \]
\end{defn}
\begin{lem}\label{lem:SliceConn}
For $V \in RO(G)$, the spectrum $S^V \wedge \widetilde{E}G$ is of slice $\geq \tau(V)$.  
\end{lem}
\begin{proof}
By \cite[Theorem~2.5]{HillYarnall}, $S^{V} \wedge \widetilde{E}G$ is of slice $\geq n$ if and only if the geometric fixed points $\Phi^H(S^{V} \wedge \widetilde{E}G) \in \tau_{\geq n/|H|}^{Post}$ for all $H \subset G$.  For $\widetilde{E}G$, its underlying space is contractible and its $H$-fixed point is $S^0$ whenever $H$ is a nontrivial subgroup of $G$.  Since $\Phi^HS^{V} = S^{V^H} \in \tau_{\geq \dim V^H}^{Post}$, $S^V \wedge \widetilde{E}G$ is of slice $\geq \tau(V)$.  \end{proof}

\begin{thm}\label{prop:isoslicehfp}
The map from the $RO(G)$-graded slice spectral sequence to the $RO(G)$-graded homotopy fixed point spectral sequence 
\[\begin{tikzcd}
\E_2^{s, V} \! \! = \pi_{V-s}^G P^{|V|}_{|V|} X \ar[r] \ar[d, Rightarrow] & \E_2^{s, V} \! \! =\pi_{V-s}^G F(EG_+, P^{|V|}_{|V|} X) \ar[d, Rightarrow ] \\ 
\pi_{V-s}^G X \ar[r] & \pi_{V-s}^G F(EG_+, X)
\end{tikzcd}\]
induces an isomorphism on the $\E_2$-page for pairs $(V, s)$ that satisfy the inequality 
\[\tau(V-s-1) > |V|.\]
Furthermore, this map induces a one-to-one correspondence between the differentials in this isomorphism region.  
\end{thm}
\begin{proof}
Applying the functor $F(-, P^{|V|}_{|V|} X)$ to the cofiber sequence 
\[EG_+ \longrightarrow S^0 \longrightarrow \widetilde{E}G\]
produces the cofiber sequence 
\[F(\widetilde{E}G, P^{|V|}_{|V|} X) \longrightarrow P^{|V|}_{|V|} X \longrightarrow F(EG_+, P^{|V|}_{|V|} X).\]
The long exact sequence in homotopy groups implies that the map 
\[\pi_{V-s}^G P^{|V|}_{|V|} X \longrightarrow \pi_{V-s}^G F(EG_+, P^{|V|}_{|V|} X)\]
is an isomorphism when both $\pi_{V-s}^G F(\widetilde{E}G, P^{|V|}_{|V|} X)$ and $\pi_{V-s-1}^G F(\widetilde{E}G, P^{|V|}_{|V|} X)$ are trivial.  Since $\pi_{\star}^G F(\widetilde{E}G, P^{|V|}_{|V|} X) =\pi_0^G F(S^\star \wedge \widetilde{E}G, P^{|V|}_{|V|}X)$ and $P^{|V|}_{|V|}X$ is a $|V|$-slice, it suffices to find pairs $(V, s)$ such that $S^{V-s-1} \wedge \widetilde{E}G$ is of slice greater than $|V|$.  By \cref{lem:SliceConn}, this is equivalent to $(V, s)$ satisfying the inequality $\tau(V-s-1) > |V|$. 

We will now use induction on $r$ to show that the map of spectral sequences induces a one-to-one correspondence between all the $d_r$-differentials whose source and target are both in the isomorphism region.  The base case of the induction, when $r = 1$, is trivial.  

For the induction step, suppose that the map induces a one-to-one correspondence between all the $d_{r'}$-differentials in the isomorphism region for all $r' < r$.  Let $d_r(x) = y$ be a $d_r$-differential in $\SliceSS(X)$ such that both $x$ and $y$ are in the isomorphism region.  By naturality, $y'$ (the image of $y$) must be killed by a differential of length at most $r$ in $\HFPSS(X)$.  If the length of this differential is $r$, then the source must be $x'$ (the image of $x$) and we are done.  If the length of this differential is smaller than $r$, then the induction hypothesis implies that the same differential must appear in $\SliceSS(X)$.  This would mean that $y$ is killed by a differential of length smaller than $r$, which is a contradiction.  Therefore all the $d_r$-differentials in $ \SliceSS(X)$ that are in the isomorphism region appear in $\HFPSS(X)$.  

On the other hand, let $d_r(x') = y'$ be a $d_r$-differential in $\HFPSS(X)$ such that both $x'$ and $y'$ are in the isomorphism region.  Let $x$ be the pre-image of $x'$.  By naturality, $x$ must support a differential of length at most $r$.  If this differential is of length exactly $r$, then naturality implies that the target must be $y$, the unique preimage of $y'$.  If the length is smaller than $r$, then by the induction hypothesis, $x'$ must support a differential of length smaller than $r$ as well.  This is a contradiction.  Therefore all the $d_r$-differentials in $\HFPSS(X)$ that are in the isomorphism region appear in $\SliceSS(X)$.  This completes the induction step.  
\end{proof}

\begin{rem}\rm
In the integer-graded page, let $V = t \in \mathbb{Z}$.  Let $m(G)$ be the order of the smallest nontrivial subgroup of $G$.  When $t - s \geq 1$, $\tau(t-s-1) = m(G)(t-s-1)$, and the isomorphism region in \cref{prop:isoslicehfp} is defined by the inequality 
\[m(G)(t-s-1) > t.\]
This recovers Theorem~I.9.4 in \cite{Ullman}.
\end{rem}

\begin{example}\rm
When $G = C_{2^n}$, $RO(G)$ is generated by $\{1, \sigma, \lambda_2, \ldots, \lambda_{n}\}$.  The representations $\lambda_i$ are rotations and have no $H$-fixed points when $H$ is a nontrivial subgroup of $G$.  Therefore, if we fix an element $V \in RO(G)$ of the form 
\[V = c_1 \cdot \sigma + c_2\cdot \lambda_2 + c_3\cdot \lambda_3 + \cdots + c_n\cdot \lambda_n, \, \, \, c_i \in \mathbb{Z},\]
then ${V}^{H} = (c_1 \sigma)^H$ for all nontrivial subgroups $H \subset G$.  When $t-s-1 > |c_1|$, we have the equality $\tau(V + t - s -1) = 2(c_1 + t - s -1)$.  On the $(V+t-s, s)$-graded page, the isomorphism region in \cref{prop:isoslicehfp} contains pairs $(t, s)$ that satisfy the inequality 
\[2(c_1 + t-s-1) > |V|+t,\]
or equivalently 
\[s< (t-s)+ 2c_1 -2-|V|.\]
In particular, the last inequality shows that on any of the $(V + t -s, s)$-graded pages, the isomorphism region is bounded above by a line of slope 1 when $t-s \gg 0$.  
\end{example}

\begin{thm}\label{prop:isohfptate}
The map from the $RO(G)$-graded homotopy fixed point spectral sequence to the $RO(G)$-graded Tate spectral sequence induces an isomorphism on the $\E_2$-page for classes in filtrations $s >0$, and a surjection for classes in filtration $s = 0$.  Furthermore, there is a one-to-one correspondence between differentials whose source is of nonnegative filtration.
\end{thm}
\begin{proof}
The $\E_2$-page of the Tate spectral sequence for $X$ is
\[\E_{2}^{s, V} =\hat{H}^{s}(G,\pi_0(S^{-V}\wedge X)) \Longrightarrow \pi_V^G \,\widetilde{E}G \wedge F(EG_+, X),\]
and the $\E_2$-page of the homotopy fixed point spectral sequence is 
\[\E_{2}^{s, V} =H^{s}(G,\pi_0(S^{-V}\wedge X)) \Longrightarrow \pi_V^G \,F(EG_+, X). \]
By the definition of Tate cohomology, $\hat{H}^s = H^s$ when $s >0$.  Furthermore, the map $H^0 \to \hat{H}^0$ is a surjection whose kernel is the image of the norm map.  This proves the claim about the $\E_2$-page.  The proof for the one-to-one correspondence of differentials is exactly the same as the proof in \cref{prop:isoslicehfp}. 
\end{proof}

We end this section by discussing the invertibility of certain Euler classes in the Tate spectral sequence.  Recall that if $V$ is a $G$-representation such that the fixed point set $V^H$ is trivial whenever $H \subset G$ is nontrivial, then $S(\infty V)$ is a geometric model for $EG$, and $S^{\infty V}$ is a geometric model for $\widetilde{E}G$.  Therefore, for any $G$-spectrum $X$, 
\[\widetilde{E}G \wedge X \simeq S^{\infty V} \wedge X = a_V^{-1}X.\]

Specialized to the case when $G = C_{2^n}$ and $Q_8$, we see that $\widetilde{E}C_{2^n} \simeq S^{\infty \lambda_n}$ and $\widetilde{E} Q_8 \simeq S^{\infty \mathbb{H}}$.  Moreover, if $X$ is a $G$-spectrum, then the Tate spectral sequence for $X$ is the spectral sequence associated to the tower $\{\widetilde{E}G\wedge F(EG_{+},P^{\bullet}X)\}$.  This implies that the class $a_{\lambda_n}$ is invertible in all the $C_{2^n}$-Tate spectral sequences, and the class $a_{\mathbb{H}}$ is invertible in all the $Q_8$-Tate spectral sequences.

\section{The norm structure}\label{sec:NormStructure}

In this section, we give a brief summary of results for the norm structure in equivariant spectral sequences.  For more detailed discussions, see \cite[Chapter~I.5]{Ullman}, \cite[Section~4]{HHRC_4}, and \cite[Section~3.4]{MeierShiZengSegal}.  

Consider a tower 
\[\cdots \longrightarrow P^{i+1} \longrightarrow P^i \longrightarrow P^{i-1} \longrightarrow \cdots\] 
of $G$-spectra and let $\E_{*}^{*, \star}$ be the associated spectral sequence.  Set $P_n^m = \text{fib}(P^m \to P^{n-1})$ and $P_n = P_n^\infty$.  The towers that will be relevant to us in this paper are the towers for the slice spectral sequence, the homotopy fixed point spectral sequence, and the Tate spectral sequence. 

Let $H \subset G$ be a subgroup.  Suppose we have maps $N_H^G P_n \to P_{|G/H|n}$ and $N_H^G P_n^n \to P_{|G/H|n}^{|G/H|n}$ that are (up to homotopy) compatible with the maps $P_n \to P_{n-1}$ and $P_n \to P_n^n$.  This is called the \textit{norm structure}.  It induces norm maps 
\[N_H^G: \E_2^{s, V+s} \longrightarrow \E_2^{|G/H|s, \Ind_H^GV +|G/H|s}.\]
If $X$ is a commutative $G$-spectrum, then its slice spectral sequence, homotopy fixed point spectral sequence, and Tate spectral sequence all have the norm structure that is induced from the multiplication on $X$ (for the Tate spectral sequence, the norm structure exists as long as $H \neq e$, as discussed in \cite[Example~3.9]{MeierShiZengSegal}).  

The following proposition (\cite[Proposition~3.7]{MeierShiZengSegal}) is a restatement of \cite[Proposition I.5.17]{Ullman} and \cite[Theorem~4.7]{HHRC_4}.  It describes the behaviour of differentials under the norm structure. 

\begin{prop}\label{thm:normdiff}
Let $x \in \E_2(G/H)$ be an element representing zero in $\E_{r+1}(G/H)$.  Then $N_H^G(x)$ represents zero in $\E_{|G/H|(r-1)+2}(G/G)$. 
\end{prop}

In other words, \cref{thm:normdiff} states that if $x \in \E_2^{s, V+s}(G/H)$ is killed by a $d_r$-differential, then $N_H^G(x)\in \E_2^{|G/H|s, \Ind_H^G V + |G/H|s}(G/G)$ must be killed by a differential of length at most $|G/H|(r-1)+1$. 



Let $\sigma_2$ be the sign representation of $C_2$.  As an immediate consequence of Equations (\ref{eq:normavuv1}) and (\ref{eq:normavuv2}), we have the following proposition.

\begin{prop}\label{prop:normofhomologyclass}
The following equalities hold: 
\begin{align*}
N_{C_2}^{C_{2^n}}(a_{\sigma_2}) &= a_{\lambda_n}^{2^{n-2}},\\ 
N_{C_2}^{C_{2^n}}(u_{2\sigma_2}) &= \frac{u_{\lambda_n}^{2^{n-1}}}{u_{2\sigma}\prod_{i = 2}^{n-1} u_{\lambda_i}^{2^{i-1}}},\\
N_{C_2}^{Q_8}(a_{\sigma_2}) &= a_\mathbb{H},\\
N_{C_2}^{Q_8}(u_{2\sigma_2}) &= \frac{u_{\mathbb{H}}^2}{u_{2\sigma_i}u_{2\sigma_j}u_{2\sigma_k}}.
\end{align*}
\end{prop}

\begin{proof}
The equalities follow from (\ref{eq:normavuv1}), (\ref{eq:normavuv2}), and the following facts about induced representations: 
\begin{align*}
\Ind_{C_2}^{C_{2^n}}(1) &= 1 + \sigma + \sum_{i = 2}^{n-1} 2^{i-2} \lambda_i, \\
\Ind_{C_2}^{C_{2^n}}(\sigma_2) &= 2^{n-2} \lambda_n, \\
\Ind_{C_2}^{Q_8}(1) &= 1 + \sigma_i + \sigma_j + \sigma_k,\\ 
\Ind_{C_2}^{Q_8}(\sigma_2) &= \mathbb{H}.
\end{align*}
\end{proof}

\begin{thm} \label{theorem:SliceElementVanishing} \hfill
\begin{enumerate}
\item The class $N_{C_2}^{C_{2^n}}(\bar{v}_h)a_{\lambda_n}^{2^{n-2}(2^{h+1}-1)}$ in the $C_{2^n}$-slice spectral sequence for $\BPCn$ is killed on or before the $\E_{2^{h+n}-2^n+1}$-page.  
\item The class $N_{C_2}^{Q_8}(\bar{v}_h)a_{\mathbb{H}}^{2^{h+1}-1}$ in the $Q_8$-slice spectral sequence for $\BPQ$ is killed on or before the $\E_{2^{h+3}-7}$-page. 
\end{enumerate}
\end{thm}
\begin{proof}
By \cref{thm:slicediffthm}, we have the differential
\[d_{2^{h+1}-1}(u_{2\sigma_2}^{2^{h-1}}) = \bar{v}_h a_{\sigma_2}^{2^{h+1}-1}\]
in the $C_2$-slice spectral sequence for $i_{C_2}^*\BPCn$ and $i_{C_2}^*\BPQ$.  Our claims follow by applying \cref{thm:normdiff} and the equations in \cref{prop:normofhomologyclass} to $(H, G, x, r) = (C_2, C_{2^n}, \bar{v}_h a_{\sigma_2}^{2^{h+1}-1}, 2^{h+1}-1)$ and $(C_2, Q_8, \bar{v}_h a_{\sigma_2}^{2^{h+1}-1}, 2^{h+1}-1)$.
\end{proof}

\section{Vanishing in the Tate spectral sequence}\label{sec:TateVanishing}
By the work of Hahn and Shi \cite{realorientationDanny}, the Lubin--Tate theory $E_h$ admits an equivariant orientation.  More specifically, for $G \subset \mathbb{G}_h$ a finite subgroup, there is a $G$-equivariant map from $\BPG$ to $E_h$.  Furthermore, this $G$-equivariant map factors through $(N_{C_2}^G \bar{v}_h)^{-1} \BPG$: 
\[\begin{tikzcd}
\BPG \ar[r] \ar[d] & E_h \\ 
(N_{C_2}^G \bar{v}_h)^{-1} \BPG \ar[ru]
\end{tikzcd}\]
This equivariant orientation induces the following diagram of spectral sequences: 
\[\begin{tikzcd}
\SliceSS(\BPG) \ar[r] & \HFPSS(\BPG) \ar[r] \ar[d]& \TateSS(\BPG) \ar[d] \\
& \HFPSS(E_h) \ar[r] & \TateSS(E_h).
\end{tikzcd}\]

\begin{thm}\label{theorem:TateVanishing}
For any height $h$ and any finite subgroup $G \subset \mathbb{G}_h$, all the classes in the $RO(G)$-graded Tate spectral sequence for $E_h$ vanish after the $\E_{N_{h, G}}$-page. Here, $N_{h, G}$ is defined as in \cref{def:NhG}.
\end{thm}

In order to prove ~\cref{theorem:TateVanishing}, we will first prove the following lemmas: 

\begin{lem}\label{lem:psylowTate}
Let $K$ be a finite group and $H \subset K$ a $2$-Sylow subgroup.  For a $2$-local $K$-spectrum $X$, if all the classes in the $RO(H)$-graded Tate spectral sequence for $X$ vanish after the $\E_r$-page, then all the classes in the $RO(K)$-graded homotopy fixed point spectral sequence for $X$ will also vanish after the $\E_r$-page.
\end{lem}
\begin{proof}
The restriction and transfer maps induce the following maps of spectral sequences: 
\[\begin{tikzcd}
K\text{-}\TateSS(X)\ar[r, "\text{res}"] & H\text{-}\TateSS(X) \ar[r, "\text{tr}"] & K\text{-}\TateSS(X).
\end{tikzcd}\]
The composition map $\text{tr} \circ \text{res}$ is the degree-$|K/H|$ map.  Since $|K/H|$ is coprime to $2$ and $X$ is $2$-local, the composition $\text{tr} \circ \text{res}$ is an isomorphism.  This exhibits the $RO(K)$-grated Tate spectral sequence as a retract of the $RO(H)$-graded Tate spectral sequence.  The statement of the lemma follows.  
\end{proof}

\begin{lem}\label{lem:TateVanishingDBPG} \hfill
\begin{enumerate}
\item At height $h=2^{n-1}m$, the unit class in the $RO(C_{2^n})$-graded Tate spectral sequence for $(N_{C_2}^{C_{2^n}} \bar{v}_h)^{-1} \BPCn$ must be killed on or before the $\E_{2^{h+n}-2^n+1}$-page. 
\item At height $h=4k-2$, the unit class in the $RO(Q_8)$-graded Tate spectral sequence for $(N_{C_2}^{Q_8} \bar{v}_h)^{-1} \BPQ$ must be killed on or before the $\E_{2^{h+3}-7}$-page.
\end{enumerate}
\end{lem}
\begin{proof}
For $G = C_{2^n}$ and $Q_8$, consider the map from the $C_2$-slice spectral sequence for $i_{C_2}^*\BPG$ to the $C_2$-Tate spectral sequence for $i_{C_2}^*\BPG$.  \cref{thm:slicediffthm}, combined with the isomorphisms in \cref{prop:isoslicehfp} and \cref{prop:isohfptate}, shows that we have the differential 
\[d_{2^{h+1}-1}(u_{2\sigma_2}^{2^{h-1}})=\bar{v}_h a_{\sigma_2}^{2^{h+1}-1}\]
in the $C_2$-Tate spectral sequence for $i_{C_2}^*\BPG$.  Since $a_{\sigma_2}$ is invertible, after further inverting $\bar{v}_h$, we have the differential 
\[d_{2^{h+1}-1}(\bar{v}_h^{-1}u_{2\sigma_2}^{2^{h-1}}a_{\sigma_2}^{1-2^{h+1}})= 1\]
in the $C_2$-Tate spectral sequence for $i_{C_2}^* (N_{C_2}^{G} \bar{v}_h)^{-1} \BPG$.  Our claims now follow by applying \cref{thm:normdiff} to $(H, G, x, r) = (C_2, C_{2^n}, 1, 2^{h+1}-1)$ and $(C_2, Q_8, 1, 2^{h+1}-1)$.
\end{proof}

\begin{proof}[Proof of \cref{theorem:TateVanishing}]
Let $K = G \cap \mathbb{S}_h$, and let $H$ be a 2-Sylow subgroup of $K$.  By the classification of the finite subgroups of $\mathbb{S}_h$, $H$ is isomorphic to either $C_{2^n}$ or $Q_8$.  We have the equality $N_{h, G} = N_{h, K} = N_{h, H}$ by \cref{def:NhG}.  The $H$-equivariant map
\[(N_{C_2}^H \bar{v}_h)^{-1}\BPH \longrightarrow E_h\]
induces a map of the corresponding Tate spectral sequences.  By naturality and \cref{lem:TateVanishingDBPG}, the unit class in the $H$-Tate spectral sequence for $E_h$ is killed on or before the $\E_{N_{h, H}}$-page.  The multiplicative structure implies that all the classes in the $H$-graded Tate spectral sequence for $E_h$ vanish after the $\E_{N_{h, H}}$-page. By \cref{lem:psylowTate}, the same statement holds for $K$ since $H$ is a $2$-Sylow subgroup of $K$.

To extend this from $K$ to $G$, note that the quotient group $G/K$ can be identified as a subgroup of the Galois group $\Gal(k/\mathbb{F}_2)$ through the inclusion $G \rightarrow \mathbb{G}_h$. Let $k' = k^{G/K}$.  The arguments shown in \cite[Lemma 1.32, Lemma 1.37, and Remark 1.39]{IrinaTopologicalresolution} imply that the $G$-Tate spectral sequence for $E_h$ is a base change from $\WW(k')$ to $\WW(k)$ of the $K$-Tate spectral sequence for $\TateSS(E_h)$.  This means there is an isomorphism 
\[\WW(k) \otimes_{\WW(k')} \hat{H}^*(G,{E_h}_*) \stackrel{\cong}{\longrightarrow} \hat{H}^*(K,{E_h}_*)\]
on the $\E_2$-page, and all the differentials in the $K$-Tate spectral sequence are the $\mathbb{W}(k)$-linear extensions of those in the $G$-Tate spectral sequence.  Consequently, the theorem statement also holds for $G$. 
 


\end{proof}

\begin{rem} \rm \label{rem:TateVanishingModule} 
If $M$ is a $(N_{C_2}^G \bar{v}_h)^{-1}\BPG$-module, its Tate spectral sequence will also be a module over the Tate spectral sequence for $(N_{C_2}^G \bar{v}_h)^{-1}\BPG$.  The same proof as the one used in \cref{theorem:TateVanishing} will apply to show the same vanishing results in the Tate spectral sequence for $M$. 
\end{rem}
   

\section{Horizontal vanishing lines in the homotopy fixed point spectral sequence}\label{sec:HFPSSVanishing}

The vanishing of the Tate spectral sequence (\cref{theorem:TateVanishing}) leads to the existence of strong horizontal vanishing lines in the homotopy fixed point spectral sequences of Lubin--Tate theories.

\begin{thm}\label{thm:vanishingallfinite}
For any height $h$ and any finite subgroup $G \subset \mathbb{G}_h$, there is a strong horizontal vanishing line of filtration $N_{h,G}$ in the $RO(G)$-graded homotopy fixed point spectral sequence for $E_h$. 
\end{thm}

\begin{lem}\label{lem:psylow}
Let $K$ be a finite group and $H \subset K$ a $2$-Sylow subgroup.  For a $2$-local $K$-spectrum $X$, if the $RO(H)$-graded homotopy fixed point spectral sequence for $X$ has a vanishing line $\mathcal{L}_H$, then the $RO(K)$-graded homotopy fixed point spectral sequence for $X$ will also have $\mathcal{L}_H$ as a vanishing line. 
\end{lem}
\begin{proof}
The proof is analogous to that of \cref{lem:psylowTate}.  The restriction and transfer maps induce the following maps of spectral sequences: 
\[\begin{tikzcd}
K\text{-}\HFPSS(X) \ar[r,"\text{res}"] & H\text{-}\HFPSS(X) \ar[r, "\text{tr}"]& K\text{-}\HFPSS(X).
\end{tikzcd}\]
The composition map $\text{tr} \circ \text{res}$ is the degree-$|K/H|$ map.  Since $|K/H|$ is coprime to $2$ and $X$ is $2$-local, the composition $\text{tr} \circ \text{res}$ is an isomorphism.  This implies that $K\text{-}\HFPSS(X)$ is a retract of $H\text{-}\HFPSS(X)$.  It follows that the vanishing line in $H$-$\HFPSS(X)$ will force the same vanishing line in $K$-$\HFPSS(X)$. 
\end{proof}

\begin{proof}[Proof of \cref{thm:vanishingallfinite}]
Let $K = G \cap \mathbb{S}_h$, and let $H$ be a 2-Sylow subgroup of $K$.  Note that $N_{h, G} = N_{h, H}$ by \cref{def:NhG}.  By \cref{lem:psylow} and \cite[Lemma 1.32, Lemma 1.37, and Remark 1.39]{IrinaTopologicalresolution}, it suffices to prove the that the statement holds for $H$.  

Consider the map 
\[H\text{-}\HFPSS(E_h) \longrightarrow H\text{-}\TateSS(E_h).\]
By \cref{prop:isohfptate}, this map induces an isomorphism of classes above filtration 0 and a one-to-one correspondence of differentials whose sources are in non-negative filtrations.  

By \cref{theorem:TateVanishing}, all the classes in the Tate spectral sequence vanish after the $\E_{N_{h,H}}$-page.  In particular, this implies that the longest differential is of length at most $N_{h,H}$, and any class of filtration at least $N_{h, H}$ must die from a differential whose source and target both have nonnegative filtrations.  Combined with the isomorphism in \cref{prop:isohfptate}, this implies that the homotopy fixed point spectral sequence collapses after the $\E_{N_{h, H}}$-page, and there is a strong horizontal vanishing line of filtration $N_{h, H}$.  
\end{proof}

\begin{cor} \label{rem:HFPSSVanishingModule}
For any $(N_{C_2}^G \bar{v}_h)^{-1} \BPG$-module $M$, there is a strong horizontal vanishing line of filtration $N_{h, G}$ in the $RO(G)$-graded homotopy fixed point spectral sequence for $M$.
\end{cor}
\begin{proof}
By \cref{rem:TateVanishingModule}, the proof is the same as the proof of \cref{thm:vanishingallfinite}.  
\end{proof}

\begin{cor}\label{cor:uniformvanishing}
For any $K(h)$-local finite spectrum $Z$, the homotopy fixed point spectral sequence 
\[H^s(G, E_tZ) \Longrightarrow \pi_{t-s}(E^{hG} \wedge Z)\]
has a strong horizontal vanishing line of filtration $N_{h, G}$. 
\end{cor}

\begin{rem} \rm
The existence of concrete strong horizontal vanishing lines (as given by \cref{thm:vanishingallfinite}) is very useful for equivariant computations (see discussion after \cref{thm:introThm1} in 
\cref{subsec:1.1}). In \cite{DuanKongLiLuWangQ8}, \cref{thm:vanishingallfinite}, combined with the equivariant structures present in the homotopy fixed point spectral sequence, is utilized to compute $E_2^{hG_{24}}$.  The authors also believe that \cref{thm:vanishingallfinite} can be employed to establish new $RO(G)$-graded periodicities for $E_h$. 
\end{rem}

\begin{example}\rm
When $G= C_2$ and at all heights $h$, there is a $d_{2^{h+1}-1}$-differential in the $C_2$-homotopy fixed point spectral sequence for $E_h$, and there is a nonzero class $\bar{v}_h^2 a_\sigma^{2^{h+1}-2}$ in bidegree $(2^{h+1}-2, 2^{h+1}-2)$.  Therefore, the vanishing line in \cref{thm:vanishingallfinite} is sharp for $E_h^{hC_2}$. 
\end{example}

\begin{example} \rm
The computations in \cite{HHRC_4} implies that in the $RO(C_4)$-homotopy fixed point spectral sequence for $E_2$, there exists a $d_{13}$-differential 
\[d_{13}(N_{2}^{4}(\bar{t}_1)^5 u_{4\lambda}u_{4\sigma}a_{\lambda}a_{\sigma} )= N_{2}^{4}(\bar{t}_1)^8 u_{8\sigma} a_{8\lambda}\]
(where we let $\lambda = \lambda_2$ and $N_2^4(-)= N_{C_2}^{C_4}(-)$ for convenience).  Moreover, the class $N_2^4(\bar{t}_1)^{10}u_{4\lambda}u_{10\sigma}a_{6\lambda}$ in bidegree $(28, 12)$ (representing $\kappa^2$) that survives to the $\E_\infty$-page.  Therefore, our vanishing line is sharp for $E_2^{hC_4}$. 
\end{example}

\begin{example}\rm
The computations in \cite{HSWXC_4} implies that in the $RO(C_{4})$-homotopy fixed point spectral sequence for $E_4$, there is a $d_{61}$-differential 
\[d_{61}(N_2^4(\bar{t}_2)^{11}u_{16\lambda}u_{32\sigma}a_{17\lambda}a_{\sigma})=N_2^4(\bar{t}_2)^{16}u_{48\sigma}a_{48\lambda}\]
Moreover, the class $N_2^4(\bar{t}_2)^{24}N_2^4(\bar{t}_1)u_{44\lambda}u_{74\sigma}a_{30\lambda}$ in bidegree $(236, 60)$ survives to the $\E_\infty$-page.  Therefore, our vanishing line is sharp for $E_4^{hC_4}$.  
\end{example}

\begin{example}\rm
Consider the $RO(Q_8)$-homotopy fixed point spectral sequence for $E_2$.  \cref{thm:vanishingallfinite} implies that there is a strong horizontal vanishing line of filtration 25.  However, the actual vanishing line is of filtration 23.  More specifically, by Bauer's computation \cite{bauertmf}, there is a $d_{23}$-differential 
\[d_{23}(\eta \Delta^5)=\bar{\kappa}^6,\]
where $\bar{\kappa}$ is represented by the class $g$ in \cite{bauertmf}. This implies that in the Tate spectral sequence, there is a $d_{23}$-differential 
\[d_{23}(\eta \Delta^{5}\bar{\kappa}^{-6})=1.\] 
By the same argument as the one given in the proof of \cref{thm:vanishingallfinite}, the sharpest vanishing line in the homotopy fixed point spectral sequence is of filtration 23.  The bounds given in \cref{thm:vanishingallfinite} for $Q_8$ has been improved in \cite{DuanKongLiLuWangQ8} to account for the sharpness in this case.
\end{example}

\section{Horizontal vanishing lines in the slice spectral sequence}\label{sec:VanishingSlice}
We will now prove explicit horizontal vanishing lines for the slice spectral sequences of $(N_{C_2}^G \bar{v}_h)^{-1}\BPG$-modules.  

\begin{thm}\label{thm:sliceVanishing}
When $G = C_{2^n}$ or $Q_8$, the $RO(G)$-graded slice spectral sequence for any $(N_{C_2}^G \bar{v}_h)^{-1}\BPG$-module $M$ admits a horizontal vanishing line of filtration $N_{h, G}$.  
\end{thm}
\begin{lem}\label{lem:MUGCofree}
When $G = C_{2^n}$ or $Q_8$, any $(N_{C_2}^G \bar{v}_h)^{-1}\BPG$-module is cofree.  
\end{lem}
\begin{proof}
By \cite[Corollary~10.6]{HHR}, we need to show that $\Phi^{H} (N_{C_2}^G \bar{v}_h)^{-1}\BPG$ is contractible for all non-trivial $H \subset G$.  To do so, it suffices to check that $\Phi^H(N_{C_2}^G \bar{v}_h) = 0$ for all nontrivial $H \subset G$.  Recall that $\bar{v}_h \in \pi_{(2^h-1)\rho_2}^{C_2} \BPG$ is defined to be the composition 
\[\begin{tikzcd}S^{(2^{h}-1)\rho_2} \ar[r, "\bar{v}_h"] & \BPR  \ar[r]& i_{C_2}^*\BPG.
\end{tikzcd}\]
The claim now follows from the fact that for the class $\bar{v}_h \in \pi_{(2^h-1)\rho_2}^{C_2} \BPR$, $\Phi^{C_2} (\bar{v}_h) = 0$ and therefore
\[\Phi^{H} (N_{C_2}^H \bar{v}_h) = \Phi^{C_2} (\bar{v}_h) = 0\] 
for all nontrivial $H \subset G$. 
\end{proof}

\begin{proof}[Proof of \cref{thm:sliceVanishing}]
Since the spectrum $M$ is cofree by \cref{lem:MUGCofree}, both the slice spectral sequence and the homotopy fixed point spectral sequence converge to the same homotopy groups: 
\[\begin{tikzcd}
\SliceSS(M) \ar[r] \ar[d, Rightarrow] & \HFPSS(M) \ar[d, Rightarrow]\\ 
\pi_\star^{G} M \ar[r, "="] & \pi_\star^{G} F(E{G}_+, M).
\end{tikzcd}\]
Consider a class $x$ on the $\E_2$-page of the slice spectral sequence.  We claim that if the filtration of $x$ is at least $N_{h, G}$, then $x$ cannot survive to the $\E_\infty$-page.  This is because if $x$ survives to represent an element in $\pi_\star^{G} M$, then there must be a class $y$ on the $\E_2$-page of the homotopy fixed point spectral sequence that also survives to represent the same element in 
\[\pi_\star^{G} F(E{G}_+, M) = \pi_\star^{G} M.\]
Moreover, the filtration of $y$ must be at least the filtration of $x$, which is $\geq N_{h, G}$.  This is a contradiction because by \cref{rem:HFPSSVanishingModule}, there is a strong horizontal vanishing line of filtration $N_{h, G}$ in the homotopy fixed point spectral sequence. 
\end{proof}


\section{\texorpdfstring{$E_h^{hG}$}{text}-orientation of real vector bundles}\label{section:application}

In this section, we will use the strong vanishing lines established in \cref{thm:vanishingallfinite} to give an upper bound for $\Theta(h, G)$, the smallest number $d$ such that the $d$-fold direct sum of any real vector bundle is $E_h^{hG}$-orientable.

\begin{defn}\label{defn:orientation}\rm
Let $E$ be a multiplicative cohomology theory with multiplication $\mu_E: E \wedge E \to E$, and $\xi$ a virtual $k$-dimensional real vector bundle over a space $X$.  Denote the Thom spectrum of $\xi$ by $M\xi$.  An \textit{$E$-orientation} for $\xi$ is a class $u: M\xi \rightarrow \Sigma^k E$ (also called a \textit{Thom class}) such that for any map $f:Y\rightarrow X$, the pull-back $u_{f^*(\xi)}: Mf^*(\xi)\to M\xi \to \Sigma^k E$ induces an equivalence 
\begin{equation}\label{eq:CohEquiv}
\begin{tikzcd} F(\Sigma^k Y_{+},E) \ar[r, "\simeq"] &F(Mf^*(\xi),E), \end{tikzcd} \end{equation}
where (\ref{eq:CohEquiv}) is defined by sending a map $g: \Sigma^k Y_+ \to E$ to the composition 
\begin{multline*} Mf^*(\xi) = S^0 \wedge Mf^*(\xi) \xrightarrow{\iota_E \wedge \id}  E \wedge Mf^*(\xi) \xrightarrow{\id \wedge \Delta}  E \wedge Mf^*(\xi) \wedge Y_+  \xrightarrow{\id \wedge u_{f^*(\xi)} \wedge \id}  E \wedge \Sigma^k E \wedge Y_+ \\\xrightarrow{\mu_E \wedge \id}  E \wedge \Sigma^k Y_+ \xrightarrow{\id \wedge g}  E \wedge E \xrightarrow{\mu_E}  E. 
\end{multline*}
Here, $\Delta: Mf^*(\xi) \to Mf^*(\xi) \wedge Y_+$ is the Thom diagonal map.  
\end{defn}
\begin{rem}\rm
If $\xi$ is $E$-oriented, then the equivalence (\ref{eq:CohEquiv}) induces a Thom isomorphism 
\[\begin{tikzcd} E^{*-k}(Y_+) \ar[r, "\cong"] &E^{*}(Mf^*(\xi))\end{tikzcd}\]
for any map $f$.  In particular, when $f$ is the identity map, there is a Thom isomorphism 
\[\begin{tikzcd} E^{*-k}(X_+) \ar[r, "\cong"] &E^{*}(M\xi)\end{tikzcd}.\]
\end{rem}

Note that it follows immediately from \cref{defn:orientation} that for any $E$-oriented bundle $\xi$, its pull back bundle $f^*(\xi)$ is also $E$-oriented.  Our definition also recovers the classical definition of orientations.  More precisely, if we take $Y$ to be a point, then the Thom space of the pull back is $S^k$, and the restriction of the Thom class $u$ under the map 
\[E^k(Th(\xi)) \longrightarrow E^k(S^k)\]
is an $E^*$-module generator for the free rank one module $E^*(S^k)$.

For $X$ a non-equivariant spectrum, we can treat it as a $G$-spectrum equipped with the trivial $G$-action.  We have the equivalence 
\[F(EG_+, F(X, E_h))^G \simeq F(X, F(EG_+, E_h))^G \simeq F(X, F(EG_+, E_h)^G) = F(X, E_h^{hG}).\]
This equivalence allows us to use the homotopy fixed point spectral sequence to compute $(E_h^{hG})^*(X)$.  The $\E_2$-page of this homotopy fixed point spectral sequence is 
\[\E_2^{s, t} = H^s(G; E_h^tX) \Longrightarrow (E_h^{hG})^{t+s}(X).\]

Let $\gamma$ be the universal bundle on $BO$ (of virtual dimension zero).  The direct sum operation on bundles over $BO$ induces a multiplication map $m_2: BO\times BO\rightarrow BO$, which can be extended to form an $\mathbb{E}_{\infty}$-structure. Following the approach in \cite[Lemma 1.9]{May72}, we recursively define $m_k=m_2\circ(\text{id}\times m_{k-1})$.  Moreover, we will define $\Delta_n$ to be the diagonal map $BO\rightarrow BO\times \dots \times BO$ ($n$-copies), and denote the composition map $m_n\circ \Delta_n: BO\rightarrow BO$ by $[n]$. 

Let $n\gamma$ denote the pullback bundle $[n]^* \gamma$.  Following Kitchloo--Wilson \cite{KitchlooWilson}, we will denote the Thom spectrum of $n \gamma$ by $MO[n]$.  We set $MO[0]=S^0$ and define 
\[\Pi MO:=\bigvee\limits_{k\geq 0} MO[2k].\]

\begin{lem}\label{lem:multistructure}
The homotopy fixed point spectral sequence for $(E_h^{hG})^*(\Pi MO)$ is a multiplicative spectral sequence whose multiplication is commutative. 
\end{lem}

\begin{proof}
In order to ensure that the homotopy fixed point spectral sequence has a multiplicative structure, it suffices to construct a $G$-equivariant map
\[
\varphi: F(\Pi MO,E_h)\wedge F(\Pi MO,E_h)\rightarrow F(\Pi MO,E_h).
\]
We will first construct a map $\Delta: \Pi MO\rightarrow \Pi MO \wedge \Pi MO$.  Once we have constructed $\Delta$, the desired map $\varphi$ will be induced from $\Delta$ by the following composition: 
\begin{equation*}
\begin{aligned}
F(\Pi MO, E_h)\wedge F(\Pi MO, E_h) \longrightarrow F(\Pi MO\wedge \Pi MO, E_h\wedge E_h)&\xrightarrow{\mu_*} F(\Pi MO\wedge \Pi MO, E_h)\\
&\xrightarrow{\Delta^*}F(\Pi MO,E_h).
\end{aligned}
\end{equation*}
Here, $\mu: E_h\wedge E_h\rightarrow E_h$  is the multiplication map.  Consider the maps
\[
BO\xrightarrow{\Delta_2} BO\times BO\xrightarrow{[2i]\times [2j]} BO\times BO \xrightarrow[]{m_2} BO.
\]
These maps induce a map of the corresponding Thom spectra
\begin{equation*}
Th([2i]^*\gamma \oplus [2j]^*\gamma)\longrightarrow MO[2i]\wedge MO[2j].
\end{equation*}
The swap map $\tau: BO\times BO\rightarrow BO\times BO$ induces the following commutative diagram of Thom spectra:
\begin{equation}\label{eq:internalexternal}
\begin{tikzcd}
   Th([2i]^*\gamma \oplus [2j]^*\gamma) \arrow[d] \ar[r] & MO[2i] \wedge MO[2j] \ar[d] \\
  Th([2j]^*\gamma \oplus [2i]^*\gamma) \ar[r] & MO[2j] \wedge MO[2i]. 
    \end{tikzcd}
\end{equation}
Since $BO$ is an $\mathbb{E}_{\infty}$-space, there is a homotopy from  $m_{2i+2j}\circ \Delta_{2i+2j}$ to $m_2\circ (m_{2i}\times m_{2j})\circ\Delta_{2i+2j}$. This produces an equivalence from $MO[2i+2j]$ to $Th([2i]^*\gamma \oplus [2j]^*\gamma)$. Composing this with the map of Thom spectra above, we obtain a map
\begin{equation*}
MO[2i+2j]\rightarrow MO[2i]\wedge MO[2j].
\end{equation*}
By fixing $n$ and combining these maps for all pairs $(i, j)$ such that $i+j = n$, we obtain a map 
\[
MO[2n]\longrightarrow \bigvee\limits_{2i+2j=2n}MO[2i]\wedge MO[2j].
\]
Taking the wedge sum of all such maps for all $n \geq 0$ produces the map $\Delta$:
\[\Delta\colon \Pi MO = \underset{n\geq 0}{\bigvee}MO[2n]\longrightarrow \bigvee\limits_{n\geq 0}\left(\bigvee\limits_{2i+2j=2n}MO[2i]\wedge MO[2j]\right) = \Pi MO \wedge \Pi MO.\]

In order to show that the multiplication on the homotopy fixed point spectral sequence is commutative, it suffices to show that $\Delta$ is co-commutative up to homotopy. Since $BO$ is an $\mathbb{E}_{\infty}$-space, the following diagram commutes up to homotopy:
 
  \begin{center}
    \begin{tikzcd}
       m_{2i+2j}\circ \Delta_{2i+2j}\arrow[r]\arrow[rd] &m_2\circ( m_{2i}\times m_{2j})\circ \Delta_{2i+2j}\arrow[d]\\
 & m_2\circ(m_{2j}\times m_{2i}) \circ \Delta_{2i+2j}.
 \end{tikzcd}
 \end{center} 
Combining the induced homotopy commutative diagram of Thom spectra and diagram~(\ref{eq:internalexternal}) produces the following homotopy commutative diagram of Thom spectra:
 \begin{center}
\begin{tikzcd}
    MO[2i+2j]\arrow[r]\arrow[rd] &Th([2i]^*\gamma \oplus [2j]^*\gamma) \arrow[d] \ar[r] & MO[2i] \wedge MO[2j] \ar[d] \\
 & Th([2j]^*\gamma \oplus [2i]^*\gamma) \ar[r] & MO[2j] \wedge MO[2i]. 
    \end{tikzcd}
\end{center}
It follows from this that the map $\Delta$ is homotopy co-commutative, and therefore $\varphi$ is homotopy commutative. 
\end{proof}



Note that since $\gamma \otimes \mathbb{C} = 2\gamma$ as real vector bundles, $2\gamma$ is $E_h$-oriented, and we have a Thom isomorphism 
\[E_h^*(MO[2]) \cong E_h^*(BO_+) \cdot u_{2}.\]
The construction of $\varphi$ in the proof of \Cref{lem:multistructure} shows that the composition map
\[
MO[2k]\longrightarrow \overbrace{MO[2]\wedge \cdots\wedge MO[2]}^k\xrightarrow{u_2\wedge\cdots\wedge u_2} \overbrace{E_h\wedge\cdots\wedge E_h}^k \xrightarrow{\mu} E_h
\]
is $u_2^k$ in $E_h^*(\Pi MO)$.  We claim that $u_2^k$ is a Thom class for $MO[2k]$.  This is because by iteratively applying adjunction and the Thom isomorphism, we have the equivalences 
\begin{align*}
F(MO[2] \wedge \cdots \wedge MO[2], E_h) &\simeq F(MO[2] \wedge \cdots \wedge MO[2], F(MO[2], E_h)) \\
&\simeq F(MO[2] \wedge \cdots \wedge MO[2], F(BO_+, E_h))\\
&\simeq F(MO[2] \wedge \cdots \wedge MO[2] \wedge BO_+, E_h)\\
&\simeq \cdots \\
&\simeq F(BO_+ \wedge \cdots \wedge BO_+, E_h),
\end{align*}
and this is given by the Thom class $u_2 \wedge \cdots \wedge u_2$.  Pulling back this Thom class via the diagonal map $BO_+ \to BO_+ \wedge \cdots \wedge BO_+$ gives $u_2^k$, and it induces the Thom isomorphism
\[E_h^*(MO[2k]) \cong E_h^*(BO_+) \cdot u_2^k.\]

\begin{thm}\label{theorem:orientationC2m}
For any height $h$ and any finite subgroup $G\subset \mathbb{G}_h$, let $K=G\cap \mathbb{S}_h$, $H$ be a $2$-Sylow subgroup of $K$, and define $d= 2\cdot|K|\cdot |H|^{\frac{N_{h,H}-1}{2}}$. Then the $d$-fold direct sum of any real vector bundle is $E_h^{hG}$-orientable.
\end{thm}
\begin{proof}
It suffices to show that for the universal bundle $\gamma$ on $BO$, its $d$-fold direct sum $d \gamma$ is $E_h^{hG}$-orientable. To show this, we will first show that $d\gamma$ is $E_h^{hK}$-orientable. 

Let $u_2: MO[2]\to E_h$ be a Thom class for the bundle $2\gamma$.  For an element $g \in K$, define $gu_2: MO[2] \to E_h$ to be the composition
\[\begin{tikzcd} gu_2: MO[2] \ar[r,"u_2"]& E_h \ar[r,"g"]& E_h. \end{tikzcd}\]
Consider the composition 
\[u_K: MO[2\cdot|K|]\xrightarrow{\Delta}\overbrace{MO[2]\wedge \cdots\wedge MO[2]}^{|K|}\xrightarrow{g_1 u_2\wedge \cdots\wedge g_{|K|}u_2} \overbrace{E_h\wedge \cdots \wedge E_h}^{|K|}\xrightarrow{\mu}E_h,\]
where $g_1$, $g_2$, $\ldots$, $g_{|K|}$ are all the elements of the group $K$.  The map $u_K$ represents an element in $H^0(K, E_h^0(MO[2\cdot|K|]))$. 

For any $k \geq 1$, the class $u_K^k \in H^0(K, E_h^0(MO[2 \cdot |K| \cdot k]))$ is a Thom class for $MO[2 \cdot |K| \cdot k]$ and there is a Thom isomorphism 
\[\begin{tikzcd}E_h^*(BO_+) \ar[r,"\cdot u_K^k"] & E_h^*(MO[2 \cdot |K| \cdot k]).\end{tikzcd}\]
If for some $k$, the class $u_K^k$ is a permanent cycle in the homotopy fixed point spectral sequence for $(E_h^{hK})^*(MO[2 \cdot |K| \cdot k])$, then the map of spectral sequences 
\begin{equation}\label{eq:Thomsquare}
\begin{tikzcd}
H^*(K,E_h^*(BO_+)) \ar[r, "\cdot u_K^k"] \ar[d, Rightarrow] & H^*(K,E_h^*(MO[2 \cdot |K| \cdot k])) \ar[d, Rightarrow] \\ 
(E_h^{hK})^*(BO_+) \ar[r] & (E_h^{hK})^*(MO[2 \cdot |K| \cdot k])
\end{tikzcd}
\end{equation}
will induce an isomorphism 
\[(E_h^{hK})^*(BO_+)\cdot u_K^k\cong (E_h^{hK})^*(MO[2 \cdot |K| \cdot k])\]
on the $\E_\infty$-page by naturality.  Moreover, for any map $f: Y \to BO$, the pullback of the class $u_K^k$, $f^*(u_K^k)$ in $H^0(K, E_h^0(Mf^*(2\cdot|K|\cdot k \gamma)))$, will also be a permanent cycle by naturality:
\[\begin{tikzcd}
H^*(K,E_h^*(MO[2 \cdot |K| \cdot k])) \ar[r] \ar[d, Rightarrow] & H^*(K, E_h^*(Mf^*(2\cdot|K|\cdot k \gamma)))\ar[d, Rightarrow]\\
(E_h^{hK})^*(MO[2 \cdot |K| \cdot k]) \ar[r] & (E_h^{hK})^*(Mf^*(2\cdot|K|\cdot k \gamma)).
\end{tikzcd}\]
Therefore, it will also induce a Thom isomorphism on the $\E_\infty$-page of the homotopy fixed point spectral sequence for $(E_h^{hK})^*(Mf^*(2\cdot|K|\cdot k \gamma))$. 

It remains to find such a $k$ so that $u_K^k$ is a permanent cycle.  The splitting map 
\[E_h^*(MO[2 \cdot |K| \cdot k]) \longrightarrow E_h^*(\Pi MO) \longrightarrow E_h^*(MO[2 \cdot |K| \cdot k])\]
shows that the homotopy fixed point spectral sequence for $(E_h^{hK})^*(MO[2 \cdot |K| \cdot k])$ is a retract of the homotopy fixed point spectral sequence for $(E_h^{hK})^*(\Pi MO)$.  Therefore, the class $u_K^k$ is a permanent cycle in the homotopy fixed point spectral sequence for $(E_h^{hK})^*(MO[2 \cdot |K| \cdot k])$ if and only if it is a permanent cycle in the homotopy fixed point spectral sequence for $(E_h^{hK})^*(\Pi MO)$.  

By \cref{lem:multistructure}, multiplication in the homotopy fixed point spectral sequence for $(E_h^{hK})^*(\Pi MO)$ is commutative.  Furthermore, only differentials of odd lengths can occur due to degree reasons, and all the classes on the $\E_2$-page with positive filtrations are $|H|$-torsion.  Since this spectral sequence is a module over the homotopy fixed point spectral sequence for $(E_h^{hK})^*(S^0)$, it has a strong horizontal vanishing line of filtration $N_{h, K} = N_{h,H}$ by \cref{thm:vanishingallfinite}.  It follows that for $k = |H|^{\frac{N_{h,H}-1}{2}}$, the class $u_K^k$ must be a permanent cycle.  This shows that if we set 
\[d = 2\cdot|K|\cdot |H|^{\frac{N_{h,H}-1}{2}},\]
then the bundle $d\gamma$ is $E_h^{hK}$-orientable. 

To show that $d\gamma$ is also $E_h^{hG}$-orientable, note that $G/K$ can be viewed as a subgroup of the Galois group $\Gal(k/\mathbb{F}_2)$ through the inclusion $G\rightarrow \mathbb{G}_h$. Similar to the argument in the proof of \cref{theorem:TateVanishing}, the map of spectral sequences
\[
H^*(G,E_h^*(BO_+)) \xrightarrow{u_K^k} H^*(G,E_h^*(MO[2 \cdot |K| \cdot k])) 
\]
is a base change of the map of spectral sequences
\[
H^*(K,E_h^*(BO_+)) \xrightarrow{u_K^k} H^*(K,E_h^*(MO[2 \cdot |K| \cdot k])).
\]
Therefore, the class $u_K^k$ is also a permanent cycle in the homotopy fixed point spectral sequence for $(E_h^{hG})^*(MO[2\cdot|K|\cdot k])$. This finishes the proof of the theorem. 
\end{proof}

\begin{rem}\rm\label{rem:orientation}
\cref{theorem:orientationC2m} shows that $\Theta(h, G) \leq 2 \cdot |K| \cdot |H|^{\frac{N_{h, H}-1}{2}}$. It is worth noting that our bound is by no means optimal, as it is established without \textit{any} explicit computations of the homotopy fixed point spectral sequence. In contrast, Kitchloo and Wilson explicitly computed $(E_h^{hC_2})^*(BO(q))$ and established that the $2^{h+1}$-fold direct sum of any real vector bundle is $E_h^{hC_2}$-orientable \cite[Theroem~1.4]{KitchlooWilson}. In this case, our bound becomes $\Theta(h, C_2) \leq 2^{2^{h}+1}$.

Our primary goal in this section is to emphasize the existence of a concrete upper bound. It is important to highlight that our bound is derived based on the presence of a strong horizontal vanishing line of filtration $N_{h, H}$ and the fact that all classes on the $\E_2$-page with positive filtration are $|H|$-torsion. With more detailed computational knowledge of the homotopy fixed point spectral sequence for $(E_h^{hG})^*(\Pi MO)$, there is potential to obtain a significantly improved upper bound for $\Theta(h, G)$.
\end{rem}


\bibliographystyle{alpha}
\bibliography{ref}

\end{document}